\documentclass[preprint,3p,sort]{elsarticle}

\usepackage{amsfonts}
\usepackage{amsthm}
\usepackage{graphicx}
\usepackage{epstopdf}
\usepackage{algorithmic}
\usepackage{amsmath}
\usepackage{amssymb}
\usepackage{hyperref}
\usepackage[T1]{fontenc}
\usepackage{soul,color}
\usepackage{cases}
\newtheorem{theorem}{Theorem}[section]
\newtheorem{proposition}[theorem]{Proposition}
\newtheorem{lemma}[theorem]{Lemma}

\theoremstyle{remark}
\newtheorem{remark}[theorem]{Remark}
\theoremstyle{example}

\numberwithin{equation}{section}
\begin{document}


\begin{frontmatter}

%

\title{Open Boundary Conditions for Nonlinear Initial Boundary Value Problems}

\author[sweden,southafrica]{Jan Nordstr\"{o}m}
\cortext[secondcorrespondingauthor]{Corresponding author}
\ead{jan.nordstrom@liu.se}
\address[sweden]{Department of Mathematics, Applied Mathematics, Link\"{o}ping University, SE-581 83 Link\"{o}ping, Sweden}
\address[southafrica]{Department of Mathematics and Applied Mathematics, University of Johannesburg, P.O. Box 524, Auckland Park 2006, Johannesburg, South Africa}

\begin{abstract}
We present a straightforward energy stable weak implementation procedure of open boundary conditions for nonlinear initial boundary value problems.
It simplifies previous work  and its practical implementation.
 \end{abstract}



\end{frontmatter}


\section{Introduction}
To bound nonlinear initial boundary value problems (IBVPs) with inflow-outflow boundary conditions involving non-zero data is maybe the most crucial task when aiming for  stability, see \cite{HEDSTROM1979222,Dubois198893,NYCANDER2008108,Nordstrom2013,svard2014,nordstrom2019,svard2021entropy,nordstrom2022linear} for previous efforts on this subject.
In  \cite{NORDSTROM2024_BC}, new nonlinear boundary conditions and boundary procedures leading to energy bounds were presented. 
Following \cite{nordstrom_roadmap}, the continuous procedure  was mimicked (using summation-by-parts operators and weak numerical boundary conditions) leading to provable energy stability. However, both the boundary condition and weak implementation technique in \cite{NORDSTROM2024_BC} were described in a rather complicated way. In this note we simplify the formulation and present a more user friendly procedure. 

 

\section{Short summary of  previous results}\label{sec:theory}
Following \cite{NORDSTROM2024_BC} we consider the nonlinear IBVP posed on the domain $ \Omega$ with open boundary $\partial\Omega$
\begin{equation}\label{eq:nonlin}
P U_t + (A_i(U)U)_{x_i}+A_i(U)^TU_{x_i}+C(U)U= \epsilon (D_i(U))_{x_i},  \quad t \geq 0,  \quad  \vec x=(x_1,x_2,..,x_k) \in \Omega
\end{equation}
augmented with the initial condition $U(\vec x,0)=F(\vec x)$ in $\Omega$ and  the non-homogeneous boundary condition
\begin{equation}\label{eq:nonlin_BC}
B(U) U - G(\vec x,t)=0,  \quad t \geq 0,  \quad  \vec x=(x_1,x_2,..,x_k) \in  \partial\Omega.
\end{equation}
In (\ref{eq:nonlin_BC}), $B$ is the nonlinear boundary operator and $G$ the boundary data. In (\ref{eq:nonlin}), Einsteins summation convention is used and $P$ is a symmetric positive definite (or semi-definite) time-independent matrix that defines an energy norm (or semi-norm) $\|U\|^2_P= \int_{\Omega} U^T P U d\Omega$. 
The $n \times n$ matrices $A_i,C$ and the boundary operator $B$  are smooth functions of the smooth $n$ component vector $U$. The viscous fluxes (or stresses) with first derivatives are included in  $D_i$ which satisfies $U_{x_i} ^T D_i  \geq 0$ while $\epsilon$ 
is a non-negative constant parameter.
\begin{remark}
The formulation  (\ref{eq:nonlin}) can be derived for the shallow water, Euler and Navier-Stokes equations  \cite{nordstrom2022linear-nonlinear,Nordstrom2022_Skew_Euler,nordstrom2024skewsymmetric_jcp}. It enables the direct use of the Green-Gauss theorem leading to an energy-rate as shown in \cite{NORDSTROM2024_BC}. However, we stress that the new boundary treatment that we will present below can be applied to any nonlinear problem as long as it's energy rate only depends the nonlinear surface term.

\end{remark}

\subsection{The energy rates}\label{sec:theory_inviscid}
For completeness and clarity we repeat the proof (from  \cite{NORDSTROM2024_BC}) of the following proposition.
\begin{proposition}\label{lemma:Matrixrelation_complemented}
The IBVP  (\ref{eq:nonlin}) with  $C+C^T = 0$ has an energy rate that only depends on (\ref{eq:nonlin_BC}). 
\end{proposition}
\begin{proof}
The energy method applied to (\ref{eq:nonlin}) and Green-Gauss formula  leads to the following energy rate
\begin{equation}\label{eq:boundaryPart1}
\frac{1}{2} \frac{d}{dt}\|U\|^2_P +  \epsilon  \int\limits_{\Omega}U_{x_i} ^T D_i d \Omega + \oint\limits_{\partial\Omega}U^T  (n_i A_i)  \\\ U  -   \epsilon U^T (n_i D_i) \\\ ds= \int\limits_{\Omega}(U_{x_i}^T  A_i U - U^T A_i^T U_{x_i})-U^T  C U \\\ d \Omega,
\end{equation}
where $\vec n =(n_1,..,n_k)^T$ is the outward pointing unit normal. The first volume term on the right-hand side of (\ref{eq:boundaryPart1}) cancel by the skew-symmetric formulation and the second one by the condition $C+C^T = 0$.
\end{proof}
\begin{remark}
Since the right-hand side of (\ref{eq:boundaryPart1}) vanish and $ U_{x_i} ^T D_i \geq 0$ provide dissipation, Proposition \ref{lemma:Matrixrelation_complemented} show that an energy-estimate is obtained if the surface terms are bounded by external data as shown in  \cite{NORDSTROM2024_BC}. 
\end{remark}
\begin{remark}
We stress again that the new boundary treatment that we will present in Section 3, is not specifically dependent on the formulation (\ref{eq:nonlin}). It can be applied to any nonlinear problem as long as one can derive an energy rate with the structure in (\ref{eq:boundaryPart1}), i.e. which depends only on a nonlinear surface term.
\end{remark}

 \subsection{The nonlinear boundary conditions}\label{BC_theory} 
We start with  the case $\epsilon =  0$ and consider the surface term $U^T  (n_i A_i) U = U^T  \tilde A  U $
where $\tilde A$ is symmetric (the skew-symmetric part vanish). 
Next, we transform $\tilde A$ to diagonal form as $ T^T \tilde A  T =  \Lambda = diag( \lambda_i)$ giving
\begin{equation}\label{1Dprimalstab_trans_final}
U^T   \tilde A  U = W^T   \Lambda W  = (W^+)^T   \Lambda^+  W^+ + (W^-)^T   \Lambda^- W^-  \,\ \text{where} \,\ W(U) = T^{-1} U.
\end{equation}
The matrix $T$ could be the standard eigenvector matrix  {\it or} another non-singular matrix that transform $\tilde A$ to diagonal form. In the latter case, Sylvester's Criterion \cite{horn2012} guarantee that the number of positive and negative diagonal entries are the same as the number of positive and negative eigenvalues.
In the eigenvalue case $T^{-1}=T^T$, but not so in the  general  case.
$\Lambda^+ $ and  $\Lambda^-$  are the positive and negative parts of $\Lambda$ while $W^+$  and $W^-$ are the corresponding variables. Possible zero diagonal entries are included in $\Lambda^+$. 

Note that the diagonal matrix  $\Lambda=\Lambda(U)$ is solution dependent and not a priori bounded (as for linear IBVPs). For linear problems, the number of boundary conditions at a surface position is equal to the number of eigenvalues (or diagonal entries) of $\tilde A$ with the wrong (in this case negative) sign  \cite {nordstrom2020}. 
In the nonlinear case, it is more complicated (and still not known) since multiple forms of $W^T   \Lambda W$ may exist \cite{nordstrom2022linear,nordstrom2022linear-nonlinear,Nordstrom2022_Skew_Euler}.
\begin{remark}\label{Sylvester}
The nonlinear boundary procedure based on (\ref{1Dprimalstab_trans_final}) has similarities with the linear characteristic boundary procedure \cite{kreiss1970,MR436612}. Hence, with a slight abuse of notation we will refer to $\Lambda(U)$ as eigenvalues and to the variables $W(U)$ as characteristic variables, even in the general transformation case.
\end{remark}

\section{Previous and updated new general formulation of nonlinear boundary conditions}\label{nonlinear_BC_char}
The starting point for the derivation of nonlinear
 boundary conditions  
 is the diagonal form (\ref{1Dprimalstab_trans_final}) of the boundary term. First we find the formulation (\ref{1Dprimalstab_trans_final}) with a {\it minimal} number of entries in $\Lambda^-$ (as mentioned above, there might be more than one formulation). Next, we specify the transformed characteristic variables $W^-$ in terms of $W^+$ and external data and
 add a scaling possibility. The previous formulation in  \cite{NORDSTROM2024_BC} read
\begin{equation}\label{Gen_BC_form}
 BU-G=S^{-1}(\sqrt{|\Lambda^-|}W^--R \sqrt{\Lambda^+}W^+) - G =0 \,\ \text{or equivalently} \,\  \sqrt{|\Lambda^-|}W^- - R \sqrt{\Lambda^+}W^+ - SG =0
\end{equation}
where $R$ is a matrix combining values of $W^-$ and $W^+$,  $S$ an non-singular scaling matrix and $G$ is the external data. 
The boundary condition (\ref{Gen_BC_form}) where we used the notation $ |\Lambda|=diag( |\lambda_i|)$ and $ \sqrt{|\Lambda|} =diag(  \sqrt{|\lambda_i|})$ is general in the sense that it  can include all components of $W$ suitably combined by the matrices $S$ and $R$. 

We will impose the boundary conditions both strongly and weakly. For the weak imposition we introduce a lifting operator $L_C$
 that enforce the boundary conditions weakly in our governing equation (\ref{eq:nonlin}) 
 as 
\begin{equation}\label{eq:nonlin_lif}
P U_t + (A_i U)_{x_i}+A^T_i U_{x_i}+C U+L_C(\tilde \Sigma(BU-G))=\epsilon (D_i)_{x_i}.
\end{equation}
The lifting operator exist only at the surface of the domain and for two smooth vector functions  $\phi, \psi$ it satisfies $\int\limits \phi^T   L_C(\psi) d \Omega = \oint\limits \phi^T  \psi ds$
and enables development  of the numerical boundary procedure in the continuous setting \cite{Arnold20011749,nordstrom_roadmap}.
The previous weak implementation of (\ref{Gen_BC_form}) using a lifting operator was given by
\begin{equation}\label{Pen_term_Gen_BC_form}
L_C(\tilde \Sigma(BU-G))= L_C(2( I^-T^{-1})^T \Sigma ( \sqrt{|\Lambda^-|}W^--R \sqrt{\Lambda^+}W^+ - SG )),
\end{equation}
where
$\tilde \Sigma = 2( I^-T^{-1})^T \Sigma S$ is a penalty matrix parametrized by the matrix coefficient $S,\Sigma$ and $I^-W=W^-$.
In \cite{NORDSTROM2024_BC}, the following steps were required in order to determine the unknowns $R,S,\Sigma$ in (\ref{Gen_BC_form}) and (\ref{Pen_term_Gen_BC_form}). \begin{enumerate}

\item Boundedness for strong homogeneous ($G=0$) boundary conditions lead to conditions on $R$.

\item Boundedness for strong inhomogeneous ($G\neq 0$) boundary conditions lead to conditions on $S$.

\item Boundedness for weak homogeneous ($G=0$) boundary conditions lead to conditions on $\Sigma$.

\item Boundedness for weak inhomogeneous ($G\neq 0$) boundary conditions followed  from conditions 1-3.
\end{enumerate}

We will simplify the previous formulation (\ref{Pen_term_Gen_BC_form}), and start by rewriting the boundary term in (\ref{1Dprimalstab_trans_final}) as
\begin{equation}\label{1Dprimalstab_trans_final_simp}
U^T   \tilde A   \\\ U = (A^+U)^T  (A^+ U) - (A^-U)^T  (A^-U)
\end{equation}
where  $A^+=\sqrt{\Lambda^+} T^{-1}$ and $A^-=\sqrt{|\Lambda^-|} T^{-1}$. This reformulation simplifies the boundary condition (\ref{Gen_BC_form}) to
\begin{equation}\label{Gen_BC_form_simp}
 BU-G=S^{-1}(A^- - R A^+)U - G =0 \,\ \text{or equivalently} \,\  (A^--R A^+) U - SG =0
\end{equation}
Note that the new boundary condition (\ref{Gen_BC_form_simp}) is identical to the old one in (\ref{Gen_BC_form}), with the boundary operator $B$  reformulated.  Also the lifting operator in  (\ref{Pen_term_Gen_BC_form}) simplifies by using $\tilde \Sigma=2 (A^-)^T S$ to yield
\begin{equation}\label{Pen_term_Gen_BC_form_simp}
L_C(\tilde \Sigma(BU-G))= L_C(2(A^-)^T  ( (A^--R A^+) U - SG )).
\end{equation}
\begin{remark}\label{simplified}
 The new boundary procedure simplifies the previous one by {\it i)} reducing the algebra in (\ref{Gen_BC_form}) and  (\ref{Pen_term_Gen_BC_form}) to (\ref{Gen_BC_form_simp}) and (\ref{Pen_term_Gen_BC_form_simp}) respectively, {\it ii)} reducing the number of parameters from three ($R,S,\Sigma$) to two ($R,S$) and  {\it iii)} by operating with $B(U)$ explicitly on the original dependent variable $U$. 
\end{remark}
The following Lemma is the main result of this note and it  replaces Lemma 3.1 in \cite{NORDSTROM2024_BC}.

\begin{lemma}\label{lemma:GenBC}
Consider the boundary term (\ref{1Dprimalstab_trans_final_simp}), the boundary condition (\ref{Gen_BC_form_simp}) and the lifting operator (\ref{Pen_term_Gen_BC_form_simp}).

The boundary term augmented with {\bf 1. strong nonlinear homogeneous boundary conditions} is positive semi-definite if the matrix $R$ is such that
\begin{equation}\label{R_condition}
I- R^T  R  \geq 0.
\end{equation}

The boundary term augmented with {\bf 2. strong nonlinear inhomogeneous boundary conditions} is bounded by the data $G$ if
the matrix  $R$ satisfies (\ref{R_condition}) with strict inequality and the matrix $S$ is such that
\begin{equation}\label{S_condition}
I- S^T  S - (R^T S)^T (I- R^T  R)^{-1} (R^T S)  \geq 0 \quad  \text{where}  \quad (I- R^T  R)^{-1}=\sum^{\infty}_{k=0}(R^T R)^k.
\end{equation}

The boundary term augmented with {\bf 3. weak nonlinear homogeneous boundary conditions} is positive semi-definite if the matrix $R$
satisfies (\ref{R_condition}).

The boundary term augmented with {\bf 4. weak nonlinear inhomogeneous boundary conditions} is bounded by the data $G$ if
the matrix $R$ satisfies (\ref{R_condition}) with strict inequality and the matrix $S$ satisfies  (\ref{S_condition}).
\end{lemma}
\begin{remark}\label{explain_notation_nonstadard} 
We have used the notation $A \geq 0$ to indicate that the matrix $A$ is positive semi-definite above.
\end{remark}
\begin{proof} 
The proof of Lemma \ref{lemma:GenBC} is step-by-step identical to the proof of Lemma 3.1 in \cite{NORDSTROM2024_BC} by replacing $\sqrt{|\Lambda^-|}W^-,\sqrt{\Lambda^+}W^+$ with  $A^- U, A^+ U$ respectively and $\tilde \Sigma = 2( I^-T^{-1})^T \Sigma S$ with $\tilde \Sigma=2 (A^-)^T S$.
\end{proof}


\subsection{Extension to include viscous terms}\label{nonlinear_BC_visc}
The boundary terms to consider in the case when $\epsilon >  0$  are given in Proposition \ref{lemma:Matrixrelation_complemented}. The argument in the surface integral
can be reformulated into the first derivative setting.  By introducing the notation $n_i D_i/2=\tilde F$ for the viscous flux we rewrite the boundary terms as
\begin{equation}\label{newflux_form_1}
U^T  (n_i A_i) U -   \epsilon U^T (n_i D_i) =  U^T \tilde A U -  \epsilon U^T  \tilde F - \epsilon \tilde F^T  U=
\begin{bmatrix}
U \\
\epsilon \tilde F
\end{bmatrix}^T
\begin{bmatrix}
\tilde A & - I \\
 - I &  0
\end{bmatrix}
\begin{bmatrix}
U \\
 \epsilon \tilde F
\end{bmatrix}
\end{equation}
where $I$ is the identity matrix. We can now formally diagonalise the boundary terms in (\ref{newflux_form_1}), apply the boundary condition (\ref{Gen_BC_form_simp}) in combination with (\ref{Pen_term_Gen_BC_form_simp}) and  Lemma \ref{lemma:GenBC} to obtain energy bounds.

\section{Summary}\label{sec:conclusion}
We have simplified and removed unnecessary parameter dependencies in previous work on nonlinear boundary conditions in \cite{NORDSTROM2024_BC}. This will remove complexities from its practical numerical implementation.
\section*{Acknowledgment}
J.N.  was supported by Vetenskapsr{\aa}det, Sweden [no.~2021-05484 VR] and University of Johannesburg Global Excellence and Stature Initiative Funding.

\bibliographystyle{elsarticle-num}
\bibliography{References_Jan}

\begin{thebibliography}{10}
\expandafter\ifx\csname url\endcsname\relax
  \def\url#1{\texttt{#1}}\fi
\expandafter\ifx\csname urlprefix\endcsname\relax\def\urlprefix{URL }\fi
\expandafter\ifx\csname href\endcsname\relax
  \def\href#1#2{#2} \def\path#1{#1}\fi

\bibitem{HEDSTROM1979222}
G.~Hedstrom, Nonreflecting boundary conditions for nonlinear hyperbolic
  systems, Journal of Computational Physics 30~(2) (1979) 222--237.

\bibitem{Dubois198893}
F.~Dubois, P.~Le~Floch, Boundary conditions for nonlinear hyperbolic systems of
  conservation laws, Journal of Differential Equations 71~(1) (1988) 93 –
  122.

\bibitem{NYCANDER2008108}
J.~Nycander, A.~M. Hogg, L.~M. Frankcombe, Open boundary conditions for
  nonlinear channel flow, Ocean Modelling 24~(3) (2008) 108--121.

\bibitem{Nordstrom2013}
J.~Nordstr{\"o}m, Linear and Nonlinear Boundary Conditions for Wave Propagation
  Problems, Springer Berlin Heidelberg, Berlin, Heidelberg, 2013, pp. 283--299.

\bibitem{svard2014}
M.~Sv{\"a}rd, H.~{\"O}zcan, Entropy-stable schemes for the {E}uler equations
  with far-field and wall boundary conditions, Journal of Scientific Computing
  58~(1) (2014) 61--89.

\bibitem{nordstrom2019}
J.~Nordstr\"{o}m, C.~L. Cognata, Energy stable boundary conditions for the
  nonlinear incompressible {N}avier--{S}tokes equations, Mathematics of
  Computation 88~(316) (2019) 665--690.

\bibitem{svard2021entropy}
M.~Sv{\"a}rd, Entropy stable boundary conditions for the {E}uler equations,
  Journal of Computational Physics 426 (2021) 109947.

\bibitem{nordstrom2022linear}
J.~Nordström, A.~R. Winters, A linear and nonlinear analysis of the shallow
  water equations and its impact on boundary conditions, Journal of
  Computational Physics 463 111254 (2022).

\bibitem{NORDSTROM2024_BC}
J.~Nordström, Nonlinear boundary conditions for initial boundary value
  problems with applications in computational fluid dynamics, Journal of
  Computational Physics 498 112685 (2024).

\bibitem{nordstrom_roadmap}
J.~Nordstr\"{o}m, A roadmap to well posed and stable problems in computational
  physics, {Journal of Scientific Computing} 71~(1) (2017) 365--385.

\bibitem{nordstrom2022linear-nonlinear}
J.~Nordström, Nonlinear and linearised primal and dual initial boundary value
  problems: When are they bounded? how are they connected?, Journal of
  Computational Physics 455 111001 (2022).

\bibitem{Nordstrom2022_Skew_Euler}
J.~Nordström, A skew-symmetric energy and entropy stable formulation of the
  compressible {E}uler equations, Journal of Computational Physics 470 111573
  (2022).

\bibitem{nordstrom2024skewsymmetric_jcp}
J.~Nordström, A skew-symmetric energy stable almost dissipation free
  formulation of the compressible {N}avier-{S}tokes equations, Journal of
  Computational Physics 512 113145 (2024).

\bibitem{horn2012}
R.~A. Horn, C.~R. Johnson, Matrix Analysis, Cambridge University Press, 2012.

\bibitem{nordstrom2020}
J.~Nordstr\"{o}m, T.~M. Hagstrom, The number of boundary conditions for initial
  boundary value problems, SIAM Journal on Numerical Analysis 58~(5) (2020)
  2818--2828.

\bibitem{kreiss1970}
H.-O. Kreiss, Initial boundary value problems for hyperbolic systems, Commun.
  Pur. Appl. Math. 23~(3) (1970) 277--298.

\bibitem{MR436612}
B.~Engquist, A.~Majda, Absorbing boundary conditions for the numerical
  simulation of waves, Mathematics of Computation 31~(139) (1977) 629--651.

\bibitem{Arnold20011749}
D.~Arnold, F.~Brezzi, B.~Cockburn, L.~Donatella~Marini, Unified analysis of
  discontinuous {G}alerkin methods for elliptic problems, SIAM Journal on
  Numerical Analysis 39~(5) (2001) 1749--1779.

\end{thebibliography}

\end{document}